\documentclass{amsart}
\usepackage{amssymb}
\usepackage{amsbsy, amsthm, amsmath,amstext, amsopn}
\usepackage{tikz-cd}
\usepackage{multicol}
\usepackage{tikz}
\usepackage{mathtools}
\usetikzlibrary{arrows,snakes,backgrounds}
\usetikzlibrary{decorations.markings}
\usetikzlibrary{matrix}
\usetikzlibrary{graphs}

\theoremstyle{definition}
\newtheorem{theorem}{Theorem}[section]
\newtheorem{prop}[theorem]{Proposition}

\theoremstyle{definition}
\newtheorem{conj}[theorem]{Conjecture}
\newtheorem{definition}[theorem]{Definition}

\newtheorem{claim}[theorem]{Claim}

\theoremstyle{remark}
\newtheorem{rmk}[theorem]{Remark}

\newcommand{\F}{\mathbb{F}}
\newcommand{\Z}{\mathbb{Z}}

\newcommand{\R}{\mathbb{R}}
\newcommand{\Zt}{\mathbb{Z}^t}

\newcommand{\A}{\mathcal{A}}

\newcommand{\X}{\mathcal{X}}
\newcommand{\I}{\mathcal{I}}

\newcommand{\C}{\mathbb{C}}
\newcommand{\vG}{G^\vee}
\newcommand{\cO}{\mathcal{O}}
\newcommand{\cK}{\mathcal{K}}
\newcommand{\Gr}{\mathrm{Gr}}

\newcommand{\val}{\operatorname{val}}

\newcommand{\Conf}{\operatorname{Conf}}

\newcounter{Qcount}
\begin{document}

\stepcounter{Qcount}

\title{Intersection Pairings for Higher Laminations}
\author{Ian Le}
\address{Perimeter Institute for Theoretical Physics\\ Waterloo, ON N2L 2Y5}
\email{ile@perimeterinstitute.ca}

\begin{abstract}

One can realize higher laminations as positive configurations of points in the affine building \cite{Le}. The duality pairings of Fock and Goncharov \cite{FG1} give pairings between higher laminations for two Langlands dual groups $G$ and $G^{\vee}$. These pairings are a generalization of the intersection pairing between measured laminations on a topological surface.

We give a geometric interpretation of these intersection pairings. In particular, we show that they can be computed as the length of minimal weighted networks in the building. Thus we relate the intersection pairings to the metric structure of the affine building. This proves several of the conjectures from \cite{LO}

The key tools are linearized versions of well-known classical results from combinatorics, like Hall's marriage lemma, Konig's theorem, and the Kuhn-Munkres algorithm.

\end{abstract}

\maketitle

\tableofcontents

\section{Introduction}

We begin by recalling some classical theorems from combinatorics. We seek to generalize these theorems by proving linearized versions of them.

Let us recall the marriage problm. Consider the sets $S_1, S_2, \dots, S_r$. The marriage problem asks whether one can find an elements $x_i$ in each of the sets $S_i$ such that $x_i \neq x_j$ for $i \neq j$. In this situation, $x_1, \dots, x_r$ will be called a \emph{system of distinct representatives} for the sets $S_i$.

\emph{Hall's theorem} or \emph{Hall's Marriage Lemma} gives a necessary and sufficient condition for a system of distinct representatives to exist:
\begin{theorem}[Hall]
Let $S_1,S_2,\ldots,S_r$ be sets. A system of distinct representatives of the sets $S_i$ exists if and only if
$$  \left|\bigcup_{i \in I} S_i\right| \geq |I|. $$
for each subset $I \subseteq [r]$.
\end{theorem}

Hall's theorem can be phrased in terms of bipartite graphs. Let us imagine our bipartite graph with vertices divided into two sets, one set of vertices on the left and one set of vertices on the right. On the left, we may put $r$ vertices and label them $1, \dots, r$. On the right, we put all elements of the sets $S_i$. We can then draw edges connecting each vertex $i$ on the left to all the elements in $S_i$. Then finding a system of distinct representatives is the same as finding a matching in this graph using all the vertices $1, \dots, r$.

A refinement of Hall's theorem is \emph{Konig's theorem.} Instead of giving conditions for a matching using all $r$ vertices to exist, it provides a formula for the maximum size of a matching:
\begin{theorem}[Konig]
Let $S_1,S_2, \dots,S_r$ be sets. The maximum number of distinct representatives of the sets $S_i$ is
$$  \min_{I \subseteq [r]} \left( \left| \cup_{i\in I} S_i \right| + r - |I| \right). $$
\end{theorem}

This can be interpreted in terms of bipartite graphs as follows. Suppose that $I \subset r$ is chosen to minimize $\left( \left| \cup_{i\in I} S_i \right| + r - |I| \right).$ Then we can find a vertex cover of the graph (a set of vertices of the graph such that each edge is incident to one of these vertices) by taking all the vertices correpsonding to the elements in $\cup_{i\in I} S_i$ and all the vertices in $[r] \backslash I$. This is a vertex cover of size $\left| \cup_{i\in I} S_i \right| + r - |I|$. It is clear that any vertex cover must be larger than the size of any matching. Konig's theorem says that the maximal matching has the same size as the minimal vertex cover. Note that if we fix a minimal vertex cover, any maximal matching must use each vertex of the minimal cover exactly once. (Also, if we fix a maximal matching, any minimal cover must cover exactly one of the vertices of each edge in the matching.)

Finally, Konig's theorem can be used to prove a theorem of Kuhn and Munkres. Let $[c_{ij}]$ be an $n \times n$ real matrix. A transversal of $[c_{ij}]$ is a choice of $n$ entries of the matrix, one entry in each row and each column. There are $n!$ such transversals. For each transversal, we can consider the sum over that transversal. We would like to find the maximal possible value for this sum.

Suppose we have some real numbers $a_i$ and $b_j$ for $1 \leq i,j \leq n$ such that $a_i + b_j \geq c_{ij}$. We can call the set of real numbers $a_i$ and $b_j$  a \emph{potential}. It is clear that the sum of any transversal is less than or equal to $\sum_i a_i + \sum_j b_j$. Thus the sum of any transversal is less than the sum of any potential. Then we have the following theorem:

\begin{theorem}[Kuhn and Munkres]
Let $[c_{ij}]$ be a real $n\times n$ matrix. Then the maximal sum of a transversal of $[c_{ij}]$ equals the minimal sum $\sum_i a_i + \sum_j b_j$, where we require that for all $i, j$,  $a_i + b_j \geq c_{ij}$. Moreover, if the $c_{ij}$ are integers, the $a$'s and $b$'s can be taken to be integral as well.
\end{theorem}

In fact, the theorem above is actually an algorithm for constructing both the transversal and the potential. Once we find a transversal that is equal to the sum of the potential, then we know we have found both the maximal transversal and the minimal potential. Alternatively, we may interpret this as saying that the potential witnesses the fact that we have found a maximal potential.

Replacing $[c_{ij}]$ by $[-c_{ij}]$ we can prove a similar theorem that the minimal transversal is equal to the maximal sum of a potential $a_i$ and $b_j$ where $a_i + b_j \leq c_{ij}$.

Note that in the theorems of both Konig and Kuhn-Munkres, we have that the maximum of one quantity is equal to the minimum of some other quantity. The linear generalizations which we will consider will also share this feature.

We now summarize the contents of this paper. In Section 2, we will discuss linear genearlizations of Hall's and Konig's theorems. In Section 3, we will introduce and prove a linearized version of the Kuhn-Munkres theorem. In Section 4, we will a definition of higher laminations, and describe some intersection pairings. In Section 5, we show how our generalization of the Kuhn-Munkres theorem can be applied to give an interpretation of intersection pairings of higher laminations in terms of the metric geometry of the affine building.

\medskip
\noindent{\bf Acknowledgments} I am grateful to Po-Shen Loh and Aaron Pixton for past discussions about the problems discussed in this paper, to Chris Fraser for his comments on a draft of this paper, and to Jim Geelen and Kazuo Murota noting the resemblance of the theorems in this paper to those in discrete convex analysis. Finally, I thank to Evan O'Dorney for our collaboration \cite{LO}, from which many of the ideas in this paper originated.

\section{Linearization}

We will need a linearization of Hall's marriage theorem. This theorem is a specialization of Rado's theorem on matroids, and was rediscovered by Moshonkin \cite{R}, \cite{M}:
\begin{theorem}[Rado] 
Let $W_1, W_2, \dots, W_r$ be subspaces of an ambient vector space $V$. Then a system of linearly independent representatives of the $W_i$ exists if and only if for each subset $I \subseteq [r]$,
$$ \dim \sum_{i \in I} S_i \geq |I|.$$
\end{theorem}

As in Hall's marriage theorem, the condition is clearly necessary, and it turns out to be sufficient. 

One can imagine a bipartite graph with the set $[r]$ on the left, and the vector space $V$ on the right. We connect a vertex $i$ on the left with the set of vectors in $W_i$. Instead of considering the cardinality of a set on the right hand side, we are considering the dimension of a vector space.

Rado theorem has a slight generalization, which is a linear version of Konig's theorem. It will be the key input to our proof:

\begin{prop}[\cite{LO}]\label{linKonig}
If $V_1,\ldots,V_r$ are subspaces of an ambient space $V$, the maximum number of linearly independent representatives from different $V_i$'s is
  \begin{equation}\label{eq:vsKonig}
      \min_{I \subseteq [r]} \left[ \dim \left( \sum_{i\in I} V_i \right) + r - |I| \right].
  \end{equation}
\end{prop}

Suppose that $I \subset [r]$ is a subset attaining the minimum in the theorem. Then any system of linearly independent representatives must use a basis of $\sum_{i\in I} V_i$ as well as one vector from each of the spaces $V_j$ for $j \in [r] \ I$.

\section{Main Theorem}

Let us now describe our main result. Let $\cO$ be a discrete valued ring and $\cK$ its field of fractions. Although all the arguments in this section work in this generality, for the purposes of the later application to higher laminations, we will take $\cO=\F[[t]]$ and $\cK=\F((t))$ where $\F=\R$ or $\C$.

We are interested in rank $n$, $\cO$-submodules of $\cK^n$. Such full-rank $\cO$-submodules are called {\it lattices}.

Let $L_1, L_2, \dots, L_n$ be lattices in $\cK^n$. We wish to find the maximum value of

$$-\val(\det(v_1, \dots, v_n)),$$
under the condition that $v_i \in L_i$. Equivalently, we would like to find the minimum value of $\val(\det(v_1, \dots, v_n)).$ Let us then define 
$$A(L_1, L_2, \dots, L_n) = \min \{\val(\det(v_1, \dots, v_n)) | v_i \in L_i\}.$$
It is not difficult to see that the minimum value is attained for generic choices of $v_1, \dots, v_n$ using the upper semi-continuity of the valuation function.

Given any vector $w \in \cK^n$, let us define 
$$c(w,L_i) := \min\{\lambda \in \Z | t^{\lambda}w \in L_i \}.$$

It will be convenient for us to define
$$c(L,L_i) = \min \{ c(w,L_i) | w \textrm{ is a generator of } L \}.$$
Here, $w \in L$ is a generator of $L$ if and only if $w$ belongs to some set of vectors in $\cK^n$ which form a basis for $L$ as an $\cO$-module. We will say that $w$ is a \emph{tight generator} for $L$ with respect to $L_i$ if $w$ is a generator for $L$ and $c(w,L_i)=c(L,L_i)$. Let us remark that it is immediate from this definition that if $v_i \in L_i$, then $t^{-c(L,L_i)}v_i \in L$.

We claim that it is easy to see that for any lattice $L$, we have
$$\val(\det(L)) + \sum_{i=1}^n c(L,L_i) \leq A(L_1, L_2, \dots, L_n).$$
Let us briefly explain why. If
$$A(L_1, L_2, \dots, L_n) = \val(\det(v_1, \dots, v_n))$$
where $v_i \in L_i$, then $t^{-c(L,L_i)}v_i \in L$, so that $t^{-c(L,L_i)}v_i$ generate some sublattice of $L$, and so
$$\val(\det(L)) \leq \val(\det(t^{-c(L,L_1)}v_1, \dots, t^{-c(L,L_n)}v_n)) = A(L_1, L_2, \dots, L_n) - \sum_{i=1}^n c(L,L_i).$$

We have the following theorem

\begin{theorem} \label{main} There exists a lattice $L$ for which we have the equality
$$\val(\det(L)) + \sum_{i=1}^n c(L,L_i) = A(L_1, L_2, \dots, L_n).$$
In particular, this means: 
\begin{enumerate}
\item There exist $w_i \in L$ such that if we define $v_i : = t^{c(L,L_i)} w_i$ then $v_i \in L_i$. In other words, the $w_i$ are tight generators for $L$ with respect to $L_i$.
\item The $w_i$ generate $L$;
\end{enumerate}
\end{theorem}

As we shall see, this theorem is a linearized version of the Kuhn-Munkres theorem, and in order to prove it, we generalize the Kuhn-Munkres algorithm to give an algorithm for finding such an $L$.

For any choice of $w_i \in L_i$, we have that 

$$A(L_1, L_2, \dots, L_n) \leq \val(\det(w_1, \dots, w_n)).$$

Thus for any choice of $w_i \in L_i$ and any choice of a lattice $L$, we have that

$$\val(\det(L)) + \sum_{i=1}^n c(L,L_i) \leq \val(\det(w_1, \dots, w_n)).$$

The above theorem says that the minimal value of the right hand side is the maximal value of the left hand side. We may therefore interpret the theorem as saying that we have found a lattice $L$ which is a witness to the minimal value of $\val(\det(w_1, \dots, w_n)).$

\begin{rmk} Let $e_1, \dots e_n$ be a basis of $\cK^n$. Then we can recover the usual Kuhn-Munkres theorem by considering lattices of the form
$$L_i = <t^{c_{i1}}e_1, t^{c_{i2}}e_2, \dots, t^{c_{in}}e_n>.$$
For more details, see \cite{LO}.
\end{rmk}

First let us explain some heuristics which may help give a sense of the difficulties in finding $L$. For any lattice $L$, there are always tight generators for $L$ with respect to any other lattice $M$. The difficulty is finding tight generators $w_i$ for $L$ with respect to each of $L_1, \dots, L_n$ such that the $w_i$ will generate $L$. Typically, we can find tight generators $w_i$ which may be linearly independent in $\cK^n$, but they will not necessarily generate all of $L$. (And if we take the lattice $L'$ spanned by the $w_i$, the $w_i$ may not be tight generators for $L'$.)

\begin{proof} We will start with an arbitrary lattice $L$. At each step we will modify the lattice $L$ until we find one which satisfies the equality. 

First note that $n$ vectors $w_1, \dots, w_n$ will generate $L$ if and only if their images $\tilde{w_1}, \dots, \tilde{w_n}$ in $L/tL$ form a basis. 

For each $L_i$, let $W_i \subset L$ be the set of tight generators for $L$ with respect to $L_i$. Then let $\tilde{W}_i \subset L/tL$ be the projection of this set to $L/tL$. It is easy to verify that $\tilde{W}_i$ is a vector subspace of $L/tL$. We wish to choose one vector $\tilde{w_i}$ from each $\tilde{W}_i$ such that the $\tilde{w_1}, \dots, \tilde{w_n}$ form a basis of $L/tL$. We will use the terminology that for the subspaces $\tilde{W_i}$, we wish to find a \emph{system of linearly independent representatives}.

Let us explain the basic idea of the algorithm. Recall that for any $L$, we have
$$\val(\det(L)) + \sum_{i=1}^n c(L,L_i) \leq A(L_1, L_2, \dots, L_n).$$
At each stage, we will either make the left hand side bigger by some integer value, or we will increase the size of the system of linearly independent representatives. The only way this terminates is that we have a system of $n$ linearly independent representatives, which is equivalent to the equality condition for the above inequality.

Now let us describe the algorithm for finding $L$. Start with an arbitrary lattice $L$.

We have the subspace $\tilde{W}_i \subset L/tL$ for each $i \leq n$. Let us find a maximum number of linearly independent representatives $\tilde{w_i}$ for $i \in J$, where $J \subset [n]$. Then we have that by Theorem ~\ref{linKonig}, there exists a set $I \subset [n]$ such that 
$$|J| = \dim \left( \sum_{i\in I} \tilde{W}_i \right) + n - |I|.$$
Moreover, we have that among the $\tilde{w_i}$, some form a basis for $\tilde{W} := \sum_{i \in I} \tilde{W}_i$, while the rest have indices in $[n] \setminus I$. Thus we may write $J = J_1 \coprod J_2$, where $\tilde{w_i}$ for $i \in J_1$ give a basis of $\tilde{W}$, and $J_2 \subset [n] \setminus I.$

Choose any lift of $\tilde{W}$ to an $\cO$-submodule of $L$. Call this lift $W \subset L$. Now let $L' = t^{-1}W + L$.

\begin{claim} $$\val(\det(L)) + \sum_{i=1}^n c(L,L_i) \leq \val(\det(L')) + \sum_{i=1}^n c(L',L_i) .$$
\end{claim}
\begin{proof} first note that 
$$\val(\det(L')) = \val(\det(L)) - \dim W.$$
Also note that 
$$c(L',L_i) = c(L,L_i) + 1 \textrm{ for } i \in I$$
$$c(L',L_i) = c(L,L_i) \textrm{ for } i \in [n] \setminus I.$$
Note also that $\dim W \leq |I|$ with equality only when we have a complete system of linearly independent representatives. This yields the claim.
\end{proof}

Thus, at each stage, if we have not found a complete system of linearly independent representatives, we may modify $L$ while making the inequality closer. We may then find a maximum set of linearly independent representatives and iterate.
\end{proof}

\section{Buildings and Laminations}

We can now apply our main theorem, Theorem~\ref{main}, to the study of intersection pairings between higher laminations. We start by introducing the objects used to define higher laminations, the affine Grassmannian and the affine building. We will then describe some invariants $f_{ijk}^t$ of configurations in the affine Grassmannian and the affine building that are tropicalizations of invariants of configurations of flags. Our main theorem, once translated to this context, will give an interpretation of the functions $f_{ijk}^t$ in terms of the metric geometry of the building.

\subsection{Affine Grassmannian and affine buildings}

Let $G$ be one of the groups $GL_n$, $PGL_n$ or $SL_n$. Write $\vG$ for its Langlands dual group, $GL_n$, $SL_n$ or $PGL_n$, respectively. Let $\F$ be a field, which for our purposes will always be $\R$ or $\C$. Let $\cO = \F[[t]]$ be the ring of formal power series over $\F$. Note that $\cO$ is naturally a valuation ring.
% where the \emph{valuation} $\val(x)$ of an element $$x = \sum_k a_k t^k \in \F((t))$$ is the %minimum $k$ such that $a_k \neq 0$. Let $\cK = \F((t))$ be the fraction field of $\cO$.

The \emph{affine Grassmannian} for $G$ is an (ind-)scheme whose $F$-points are the set
$$\Gr(\F) = \Gr(G) = G(\cK)/G(\cO).$$

Here is a concrete description of this set when $G=GL_n$, $PGL_n$ and $SL_n$. For $G=GL_n$, a point in the affine Grassmannian is given by a \emph{lattice} in $\cK_n$ (a finitely generated, rank $n$, $\cO$-submodule of $\cK^n$). For $G=SL_n$, a point in the affine Grassmannian corresponds to a lattice which the property that this lattice has generators $v_1, \dots, v_n$ such that
$$v_1 \wedge \dots \wedge v_n=e_1 \wedge \dots \wedge e_n.$$
Here $e_1, \dots, e_n$ is the standard basis of $\cK^n$. For $G=PGL_n$, a point in the affine Grassmannian corresponds to an equivalence class of lattices up to scale: two lattices $L$ and $L'$ are equivalent if $L=cL'$ for some $k \in \cK$. In all three cases, the affine Grassmannian consists of some set of lattices. Moreover, in each case, $G(\cK)$ acts on the set of such lattices, and the stabilizer of any lattice is isomorphic to $G(\cO)$.

The affine Grassmannian has a metric naturally taking values in the dominant coweights of $G$. Recall that the coweight lattice $\Lambda$ is defined as $\mathrm{Hom}(\mathbf{G}_m,T)$. The dominant coweights are those coweights lying in the dominant cone. For example, for $G=GL_n$, the set of dominant coweights is exactly the set of $$\mu=(\mu_1, \dots, \mu_n) \in \Z^n$$
where $\mu_1 \geq \mu_2 \geq \cdots \geq \mu_n$. For $G=SL_n$ and $PGL_n$, the coweights are given by similar conditions. For $SL_n$, the dominant coweights further satsify the relation $\mu_1+\cdots+\mu_n=0$. For $PGL_n$, the dominant coweights lie in the quotient $\Z^n/(1, 1, \dots, 1)$.

Pairs of elements of $\Gr$ up to the action of $G(\cK)$ are in bijection with double cosets 
$$G(\cO) \backslash G(\cK) / G(\cO).$$ These cosets are parameterized by the set $\Lambda_+$ of dominant coweights of $G$. Let us explain further.

Fix a basis $e_1, \dots, e_n$ of $\cK^n$. For any dominant coweight $\mu=(\mu_1, \dots, \mu_m)$, we can consider the element of $\in G(\cK)$ which is given by the matrix with diagonal entries $t^{-\mu_i}$. Applying this element to the trivial lattice $<e_1, \dots, e_n>$ gives us a lattice which we will call $t^\mu$. Any two points $p$ and $q$ of the affine Grassmannian can be translated by some element of $G(\cK)$ to $t^0$ and $t^\mu$, respectively. It turns out that $\mu$ will be unique. This gives the identification of the double coset space with $\Lambda_+$.

For $p$ and $q$ as above, we will write 
$$d(p,q) = \mu$$
and say that the distance from $p$ to $q$ is $\mu$. Note that this is a non-symmetric distance function. We have the relation $$d(q,p)=-w_0 d(p,q)$$ where $w_0$ is the longest element of the Weyl group of $G$. There is a partial order on $\Lambda^+$ defined by $\lambda > \mu$ if $\lambda - \mu$ is a positive linear combination of positive co-roots). Under this partial ordering, the distance function satisfies a version of the triangle inequality. If we take this metric on the affine Grassmannion, then $G(\cK)$ acts by isometries.

We can now introduce the affine building for $G=PGL_n$, the case which is of the most interest to us. The affine building is a simplicial complex which captures the geometry of above metric on the affine Grassmannian.

The set of vertices of the affine building for $PGL_n$ are in bijection with the the points of the affine Grassmannian $\Gr(PGL_n)$. The simplices of the affine building are as follows: for any lattices $L_0, L_1, \dots, L_k$, there is a $k$-simplex with vertices at $L_0, L_1, \dots, L_k$ if and only if $$L_0 \subset L_0 \subset \cdots \subset L_k \subset t^{-1}L_0.$$
(Recall that $\Gr(PGL_n)$ consists of lattices up to scale. For this reason, we will need to choose appropriate representatives $L_0, \dots, L_k$. Note that it may be necessary to scale some of the $L_i$ in order to fulfill this condition.) The non-symmetric, coweight-valued metric we defined above descends to a metric on the affine building.

\subsection{Invariants of Flags}

We now define some functions of triples of principal affine flags as well as their tropical analogues, which are functions of triples of points in the affine Grassmannian.

We start by defining \emph{principal affine flags} for $G=SL_n$. The space of principal affine flags is parameterized by the quotient $G/U$, where $U \subset G$ is the subgroup of unipotent upper triangular matrices.  Concretely, we can specify a principal affine flag by giving an ordered basis $v_1, \dots, v_n$. These $n$ vectors determine a flag, where the $k$-dimensional subspace of this flag is spanned by $v_1, \dots , v_k$ for $k \leq n$. Additionally, the data of a principal affine flag includes the volume forms $$ v_1 \wedge \dots \wedge v_k$$ on these subspaces. Finally, we require that $$v_1 \wedge \dots \wedge v_n$$ is the standard volume form (without this requirement, we would be dealing with $GL_n$ flags). Two sets of basis vectors will determine the same principal affine flag if they give the same $k$-forms $v_1 \wedge \dots \wedge v_k$ for $k \leq n$. We will sometimes abbreviate ``principal affine flags'' by calling them \emph{principal flags}.

We would like to consider the space of three flags up to the left diagonal action of $G$:
$$G \backslash (G/U)^3$$
Let $F_1, F_2, F_3$ be three flags which are represented by bases $u_1, \dots, u_n$, $v_1, \dots, v_n$ and $w_1, \dots, w_n$ respectively. For non-negative integers $i, j, k$ such that $i+j+k=n$ we can consider the function 
$$f_{ijk}(F_1, F_2, F_3)=\det(u_1, u_2, \dots, u_i, v_1, v_2, \dots v_j, w_1, w_2, \dots, w_k).$$ 
This function is manifestly $G$-invariant. Note that if one of $i, j, k$ is $0$, these functions only depend on two of the flags. We call such functions {\em edge} functions, while the functions depending on three flags will be called {\em face} functions. For any triangulation of an $m$-gon, the edge and face functions form a cluster for the space of configurations of $m$ principal flags, which is $\A_{G,S}$ where $S$ is a disc with $m$ marked points.

The functions $f_{ijk}$ have tropical analogues, $f_{ijk}^t$, which will be functions of three points in the affine Grassmannian for $SL_n$. The functions $f_{ijk}^t$ appeared in \cite{K}, where they were called $H_{ijk}$.

Let $x_1, x_2, x_3$ be three points in the affine Grassmannian for $SL_n$, thought of as lattices. For $i, j, k$ as above with $i+j+k=n$, we consider

$$-\val(\det(u_1, \dots, u_i, v_1, \dots v_j, w_1, \dots, w_k))$$
as $u_1, \dots, u_i$ range over elements of the $\cO$-submodule $x_1$, $v_1, \dots v_j$ range over elements of $x_2$, and $w_1, \dots, w_k$ range over elements of $x_3$. Define $f_{ijk}^t (x_1,x_2,x_3)$ as the maximum value attained by this quantity.

\begin{rmk} The edge functions recover the distance between two points in the affine Grassmannian. More precisely, for $i+j=n$, we have 
$$f_{ij0}^t (x_1,x_2,x_3) = \omega_j \cdot d(x_1,x_2) = \omega_i \cdot d(x_2,x_1).$$ Here $\omega_i$ and $\omega_j$ are fundamental weights for $SL_n$.
\end{rmk}

We can extend $f_{ijk}^t$ to a function on the affine Grassmannian for $PGL_n$ in the following way. Let $x_1, x_2, x_3$ be three points in the affine Grassmannian for $PGL_n$, represented by three lattices $L_1, L_2, L_3$. For $i, j, k$ as above, we can again maximize
$$-\val(\det(u_1, \dots, u_i, v_1, \dots v_j, w_1, \dots, w_k))$$
as before. Call the resulting maximum
$$\tilde{f}_{ijk}^t (L_1,L_2,L_3).$$ Note that $\tilde{f}_{ijk}^t (L_1,L_2,L_3)$ will depend on the representative lattices $L_1, L_2, L_3$ that we chose, which are only determined up to scale. To fix this, we renormalize, defining
\begin{equation} \label{PGLfunction}
f_{ijk}^t (x_1,x_2,x_3) := \tilde{f}_{ijk}^t (L_1,L_2,L_3) + \frac{\val(\det(L_1)\det(L_2)\det(L_3))}{n}.
\end{equation}
Note that if $L_1, L_2, L_3$ have determinant $1$ (and hence correspond to points in the affine Grassmannian for $SL_n$) our definition reduces to the previous definition.

\subsection{Metric interpretation of the functions $f^t_{ijk}$}

We now give another way to compute the tropical functions $f_{ijk}^t$. Whereas $f_{ijk}^t$ was defined in a valuation-theoretic way, it turns out that it is determined by the coweight-valued metric on the affine Grassmannian (and affine building).

We need some notation first. Let $\omega_i$ be the $i$-th fundamental weight for $SL_n$: $\omega_i = (1, \dots, 1, 0, \dots, 0)$ where there are $i$ $1$'s and $n-i$ $0$'s. Recall that for any two points $p, q$ in the affine Grassmannian, $d(p,q)$ is an element of the coweight lattice for $SL_n$.

Define, for $1 \leq i \leq n-1$,
$$d_i(L,M) = \omega_i \cdot d(L,M)$$
Note that if $L, M$ are both in the affine Grassmannian for $SL_n$, then $d_i(L,M)$ is an integer for all $i$, while if they are in the affine Grassmannian for $PGL_n$, $d_i(L,M)$ may have denominator $n$.

We will extend the definition of $d_i(L,M)$ to the case when $L, M$ are both in the affine Grassmannian for $PGL_n$. We will need to view the coweight lattice of $PGL_n$ as containing the coweight lattice of $SL_n$ as an index $n$ sublattice. Recall that the coweight lattice for $SL_n$ is the subset of $\Z^n$ given by 
$$\{(x_1, \dots, x_n) |  x_1 + \cdots + x_n=0\}.$$
The coweight lattice for $PGL_n$ is given by
$$\Z^n/(1, 1, \dots, 1).$$
Any point in $\Z^n$ can be translated into the plane $x_1 + \cdots + x_n=0$ using some multiple of $(1, 1, \dots, 1)$, at the cost of possibly introducing entries in $\frac{1}{n} \Z.$ One can then easily see that the weight latice for $PGL_n$ naturally contains as an index $n$ sublattice the coweight lattice of $SL_n$.

Therefore if $p$ and $x$ are in the affine Grassmannian for $PGL_n$, viewing the coweights of $PGL_n$ in the same space as the coweights of $SL_n$, we can define pairings with $SL_n$ weights. Thus we may consider the quantities 
$$d_i(p,x) := \omega_i \cdot d(p,x).$$
These quantities lie in $\frac{1}{n} \Z$. 

\begin{theorem} \label{main} Let $x_1, x_2, x_3$ be any configuration of points in the affine Grassmannian for $PGL_n$. Then 
$$f_{ijk}^t (x_1,x_2,x_3) = \min_{p} d_i(p,x_1) + d_j(p,x_2) + d_k(p,x_3),$$
where the minimum is taken over all $p$ in the affine Grassmannian for $PGL_n$.
\end{theorem}

Note that both sides of the expression are in $\frac{1}{n} \Z$. A special case of the theorem is when $x_1, x_2, x_3$ is a configuration of points in the affine Grassmannian for $SL_n$. In this case, both sides of the expression are integers. The integrality on the left comes from a valuation, while the integrality on the right is slightly more subtle. In general, it is easy to check that the expression 
$$d_i(p,x_1) + d_j(p,x_2) + d_k(p,x_3)$$
is determined mod $1$ by $x_1, x_2, x_3$, i.e., it does not depend on $p$ when considered mod $1$. Clearly if $x_1, x_2, x_3, p$ are all in the affine Grassmannian for $SL_n$, the pairing between coweights and weights guarantees that the expression is integral, hence it is integral for any $p$.

Note that even in the case that $x_1, x_2, x_3$ are all in the affine Grassmannian for $SL_n$, the minimizing point $p$ may be in the affine Grassmannian for $PGL_n$.

We now have a description of the functions $f_{ijk}^t$ that is completely of a metric nature. Therefore the functions $f_{ijk}^t$ may also be naturally viewed as functions on configurations of points in the affine building. We observed in \cite{Le} that for \emph{positive} configurations of points in the affine building, the functions $f_{ijk}^t$ only depended on metric properties of the configuration within the building. However, we did not give an explicit formula for this dependence. We now have an explicit formula that holds for all configurations, not just positive ones.

\begin{proof} Let us now show how the theorems above follow from our main theorem. We will give the proof in the case that $G=PGL_n$, which is more general than the case when $G=SL_n$.

%because the affine Grassmannian/building for $PGL_n$ contains the affine Grassmannian/building for $SL_n$.

We wish to show that
$$f_{ijk}^t (x_1,x_2,x_3) = \min_{p} d_i(p,x_1) + d_j(p,x_2) + d_k(p,x_3).$$
First let $x_1, x_2, x_3$ be represented by lattices $L_1, L_2, L_3.$ Let us take in the main theorem the $n$ lattices
\[
  \underbrace{L_1,\ldots,L_1}_{i}, \underbrace{L_2,\ldots,L_2}_{j}, \underbrace{L_3,\ldots,L_3}_{k}
\]
The theorem gives us that there exists a lattice $L$ such that 
$$\val(\det(L)) + i \cdot c(L,L_1) + j \cdot c(L,L_1) + k \cdot c(L,L_3)= A(L_1, \dots, L_1, L_2, \dots, L_2, L_3, \dots, L_3),$$
where $A(L_1, \dots, L_1, L_2, \dots, L_2, L_3, \dots, L_3)$ is the minimal value of
$$\val(\det(u_1, \dots, u_i, v_1, \dots v_j, w_1, \dots, w_k))$$
as $u_1, \dots, u_i$ range over elements $L_1$, $v_1, \dots v_j$ range over elements of $L_2$, and $w_1, \dots, w_k$ range over elements of $L_3$.

The first thing to note is that
$$A(L_1, \dots, L_1, L_2, \dots, L_2, L_3, \dots, L_3)=-\tilde{f}_{ijk}^t (L_1,L_2,L_3).$$

The theorem tells us that there are vectors
$$u_1, \dots, u_i, v_1, \dots v_j, w_1, \dots, w_k$$
minimizing
$$\val(\det(u_1, \dots, u_i, v_1, \dots v_j, w_1, \dots, w_k))$$
such that the vectors
$$t^{-c(L,L_1)}u_1, \dots,  t^{-c(L,L_1)}u_i$$
$$t^{-c(L,L_2)}v_1, \dots,  t^{-c(L,L_2)}v_j$$
$$t^{-c(L,L_3)}w_1, \dots,  t^{-c(L,L_3)}w_j$$
all lie in $L$, and moreover are a set of generators for $L$. Moreover, these vectors are tight generators for $L$ with respect to $L_1, L_2,$ and $L_3$.

Let us unravel what this tells us. This means, for example, that if we view $d(L,L_1)$ as an element of $\Z^n$ (viewing $L$, and $L_i$ as $GL_n$ lattices temporarily), then $d(L,L_1)$ has as its first $i$ entries $-c(L,L_1)$. A simple calculation then tells us that 
$$d_i(L,L_1) = -i \cdot c(L,L_1)+\frac{i}{n}(\val(\det(L_1))-\val(\det(L))).$$
Similarly,
$$d_j(L,L_2) = -j \cdot c(L,L_2)+\frac{j}{n}(\val(\det(L_2))-\val(\det(L))).$$
$$d_k(L,L_3) = -k \cdot c(L,L_3)+\frac{k}{n}(\val(\det(L_3))-\val(\det(L))).$$

Putting this together with Equation ~\ref{PGLfunction} yields the result.

\end{proof}

It is not hard to see that the theorem has the following mild generalization. Let $i_1+i_2+\cdots+i_k=n$. For points $x_1, x_2, \dots, x_k$ in the affine building for $PGL_n$ or $SL_n$, we can define the functions
$$f_{i_1i_2\dots i_k}^t(x_1,x_2\dots,x_k).$$
Then we have that
\begin{theorem}\label{basicfn}
$$f_{i_1i_2\dots i_k}^t(x_1,x_2\dots,x_k)= \min_{p} d_{i_1}(p,x_1) + d_{i_2}(p,x_2) + \cdots + d_{i_k}(p,x_k)$$
where the minimum is taken over all $p$ in the affine Grassmannian for $PGL_n$.
\end{theorem}

\subsection{Relationship to the Duality Conjectures}

We will now specialize to the case of positive configurations of points in the affine building, which give tropical points of $\A_{G,S}$ in the case that $S$ is a disc with marked points. First we recall the definition of a positive configurations in the affine building, which is a particular type of \emph{higher lamination}.

\begin{definition} Take $m$ points of the real affine Grassmannian, $$x_1, x_2, \dots x_m.$$
This configuration of points will be called a \emph{positive configuration} of points in the affine Grassmannian if and only if we have a collection of ordered bases for each $x_i$,
$$v_{i1}, v_{i2}, \dots, v_{in}$$
such that for each triple of integers $p, q, r$, $1 \leq p < q <r \leq m$, and each triple of non-negative integers $i, j, k$ such that $i+j+k=n$,
\begin{itemize}
\item $f_{ijk}^t (x_p,x_q,x_r) = -\val(\det(v_{p1}, \dots, v_{pi}, v_{q1}, \dots v_{qj}, v_{r1}, \dots, v_{rk}))$, and moreover
\item the leading coefficient of $\det(v_{p1}, \dots, v_{pi}, v_{q1}, \dots v_{qj}, v_{r1}, \dots, v_{rk})$ is positive.
\end{itemize}
\end{definition}

We can use the same definition when $G=SL_n$ or $PGL_n$.

\begin{rmk} It is sufficient to check the above two conditions for only those triples $p, q, r$ that are vertices of a triangle in a particular triangulation of the $m$-gon. If the conditions hold in one triangulation, they hold in any other triangulation.
\end{rmk}

The duality conjectures of Fock and Goncharov concern two dual spaces $\A_{G,S}$ and $\X_{\vG,S}$, where $\vG$ is the Langlands dual group of $G$. The conjectures roughly state that the tropical points of one space parameterize a canonical basis of functions in the other space. This means that, for example, $\X_{\vG,S}(\Zt)$ parameterizes a basis of functions for $\A_{G,S}$. This bijection satisfies many compatibility relations which we will not discuss here. These conjectures have been proved in many cases by Goncharov and Shen \cite{GS2}, building on the work of Gross, Hacking, Keel and Kontsevich \cite{GHKK}.

The duality conjectures further imply that there should be a pairing between tropical spaces:
$$\X_{\vG,S}(\Zt) \times \A_{G,S}(\Zt) \rightarrow \Z.$$
Let us explain how to get this pairing. A tropical point $l \in \X_{\vG,S}(\Zt)$ corresponds to a function $f_l$ on $\A_{G,S}$ by duality conjectures. A point $l' \in \A_{G,S}(\Zt)$ arises by taking valuations of some (positive) Laurent-series valued point $x_{l'} \in \A_{G,S}(\cK)$:
$$-\val(x_{l'}) = l'.$$
Then we get a pairing $\mathcal{I}$:
$$\mathcal{I}(l,l')=-\val f_l(x_{l'}).$$
The value of $\mathcal{I}(l,l')$ is independent of the choice of the point $x_{l'}$, because $f_l$ should be a Laurent polynomial in the cluster coordinates.

Alternatively, we can perform the dual construction. A lamination $l' \in \A_{G,S}(\Zt)$ corresponds to a function $f_{l'}$ on $\X_{\vG,S}$. Then 
$$\mathcal{I}(l,l')=-\val f_{l'}(x_{l})$$
where $-\val(x_{l}) = l$ for $x_{l} \in \X_{vG,S}(\cK)$. Surprisingly, these dual constructions conjecturally give the same answer.

The pairing $\mathcal{I}$ can be interpreted as an \emph{intersection pairing} between higher laminations, once one identifies the tropical spaces $\X_{\vG,S}(\Zt)$ and $\A_{G,S}(\Zt)$ with higher laminations for the groups $\vG$ and $G$, respectively. When $G=SL_2$, $\mathcal{I}$ specializes to the usual intersection pairing between $\A$- and $\X$-laminations on a surface $S$ (\cite{FG1}).

The functions $f_{ijk}$ are particular examples of cluster variables on $\A_{G,S}$. Cluster variables on $\A_{G,S}$--in fact, cluster monomials--are conjecturally part of the canonical basis parameterized by $\X_{\vG,S}(\Zt)$. Consider a cluster for the space $\A_{G,S}$ consisting of functions $a_1, \dots, a_N$. Let the corresponding cluster $\X$-variables on the space $\X_{\vG,S}$ be $x_1, \dots, x_N$. (There is a notion of ``corresponding variable'' because $\A_{G,S}$ and $\X_{\vG,S}$ are part of dual cluster ensembles.) Then for integers $d_i \geq 0$ the cluster monomial 
$$a_1^{d_1} \cdots a_N^{d_N}$$
should correspond to the tropical point in $(d_1, \dots, d_N) \in \X_{\vG,S}(\Zt)$ in the coordinate chart $x_1, \dots, x_N$.

Theorem ~\ref{main} gives a way of computing $f^t_{ijk}(l')$ for any point $l' \in \A_{G,S}(\Zt)$. Thus it gives a geometric interpretation of the intersection pairings. For any cluster chart coming from a triangulation of $S$, the associated cluster monomials are parameterized by a cone in $\X_{\vG,S}(\Zt)$. We therfore understand the pairing between a union of cones in $\X_{\vG,S}(\Zt)$ and the whole space $\A_{G,S}(\Zt)$.

One would like to have a geometric interpretation of these intersection pairings in general, for all points of $\X_{\vG,S}(\Zt)$ and $\A_{G,S}(\Zt)$.

We would like to mention a surprising consequence. First recall that the functions $f_{ijkl}$ satisfy various identities, for example, for $i+j+k+l=n$, we have:

$$f_{ijkl} f_{i+1, j-1, k+1, l-1} = f_{i, j, k+1, l-1} f_{i+1, j-1, k, l} + f_{i+1, j, k, l-1} f_{i, j-1, k+1, l}.$$

Then we can tropicalize this to get that evaluating the three functions

$$f^t_{ijkl} + f^t_{i+1, j-1, k+1, l-1},$$
$$f^t_{i, j, k+1, l-1} + f^t_{i+1, j-1, k, l},$$ 
$$f^t_{i+1, j, k, l-1} + f^t_{i, j-1, k+1, l}$$
on four points in the affine building gives three numbers such that the two largest of these numbers are equal. This statement seems fairly non-trivial if we use the metric interpretation of the functions $f^t_{ijkl}$.

\section{Generalizations}

We explained in the previous section how our metric formula for the function $f_{ijk}^t$ gave a way to compute intersection pairings between a subset of $\X_{\vG,S}(\Zt)$ and $\A_{G,S}(\Zt)$. The key observation was that if $f$ is a cluster variable in some cluster, then giving a metric interpretation of $f^t$ gives us a way of computing some set of intersection pairings. More precisely, if $f_l$ is the function on $\A_{G,S}(\Zt)$ which corresponds to the tropical point $l \in \X_{\vG,S}(\Zt)$, then giving a metric interpretation of $f_l^t$ is the same as computing $\I(l, -)$. Our goal for this section will therefore be to give a metric formula for $f^t$ for several more instances where $f$ is a cluster variable.

One can show using the sequence of mutations for a flip and an inductive argument that the functions 
$$f_{i_1i_2\dots i_k}$$
are cluster variables. Then Theorem~\ref{basicfn} gives a metric formula for $$f_{i_1i_2\dots i_k}^t.$$

We will now show how this can be further extended to other cluster variables.

\subsection{Tropicalization of Functions}

We begin by recalling from \cite{GS} and \cite{Le} how to evaluate $f^t$ on a higher lamination for a cluster variable $f$. Our treatment will be slightly different from those papers, and will be tailored to our particular goals.

Let $\A$ denote the variety of principal affine flags, $G/U$. The space of configurations of $m$ flags, denoted $\Conf_m (\A)$, the quotient of $(G/U)^n$ by the diagonal action of $G$. 

It is well-known that the functions on $\A_G$ are naturally isomorphic to 
$$\bigoplus_{\lambda \in \Lambda_+} V_{\lambda}$$
as a $G$-representation under the left action of $G$. Moreover, it is a fact that any cluster variable $f$ in the space of functions $\mathcal{O}(\Conf_m (\A))$ is given by an invariant in
$$[V_{\lambda_1}^* \otimes V_{\lambda_2}^* \otimes \cdots \otimes V_{\lambda_m}^*]^G$$
for some set of dominant weights $\lambda_i$. Here, $V^*$ is the representation dual to $V$. We use the dual representations here for convenience.

For example, the functions $f_{ijk}$ defined previously are given by invariants in 
$$[V_{\omega_i}^* \otimes V_{\omega_j}^* \otimes V_{\omega_k}^*]^G$$

A point in $\A$ gives a compatible family of vectors $v_{\mu} \in V_{\mu}$ for all highest weights $\mu$. A point in $(G/U)^n$ gives a vectors  
$$v_{\mu_1} \otimes v_{\mu_2} \otimes \cdots \otimes v_{\mu_m} \in V_{\mu_1} \otimes V_{\mu_2} \otimes \cdots \otimes V_{\mu_m}$$
for all $m$-tuples of highest weights $(\mu_1, \mu_2, \dots, \mu_m)$. Up to the action of $G$, we get a vector in
$$[V_{\mu_1} \otimes V_{\mu_2} \otimes \cdots \otimes V_{\mu_m}]^G.$$
Then the function $f$ evaluated on a point in $\Conf_m (\A)$ is then given by contracting $f$ with the vector 
$$v_{\lambda_1} \otimes v_{\lambda_2} \otimes \cdots \otimes v_{\lambda_m}.$$

Let us now describe how to tropicalize $f$. 
A point in the affine Grassmannian of $G$ gives us not only a lattice, but a lattice in every representation of $G$. Let $x_1, \dots, x_m$ be a configuration of points in the affine Grassmannian. For each $x_i$, we have a lattice $L_i \subset V_{\lambda_i} \otimes \mathcal{K}$. We maximize the expression
$$-\val(f(v_1, v_2, \dots, v_m))$$
over vectors $v_i \in L_i$. This maximum value will give $f^t(x_1, x_2, \dots, x_m)$.

\subsection{Some examples}

Now, let us first make an easy observation. Because the outer automorphism of $SL_n$ acts on everything in sight, we have a dual statement to Theorem~\ref{basicfn}. Note that the outer automorphism of $SL_n$ takes any representation to its dual representation, and hence interchanges the weights $\omega_a$ and $\omega_{n-a}$. Now suppose that $i_1+i_2+\cdots+i_k=(k-1)n$. Then if we put $j_s=n-i_s$, we have that $j_1+j_2+\cdots+j_k = n$. We can then define the function
$$f_{i_1i_2\dots i_k}$$
which dual to the function $f_{j_1j_2\dots j_k}$. These functions also turn out to be cluster variables \cite{GS2}, \cite{Le2}.

For points $x_1, x_2, \dots, x_k$ in the affine building for $PGL_n$ or $SL_n$, we can define the functions
$$f_{i_1i_2\dots i_k}^t(x_1,x_2\dots,x_k).$$
by the procedure given in the previous section. Then we have that
\begin{theorem}
$$f_{i_1i_2\dots i_k}^t(x_1,x_2\dots,x_k)= \min_{p} d_{i_1}(p,x_1) + d_{i_2}(p,x_2) + \cdots + d_{i_k}(p,x_k)$$
where the minimum is taken over all $p$ in the affine Grassmannian for $PGL_n$.
\end{theorem}

This is an easy theorem given the above discussion, but it motivates the first example that goes beyond the results of this paper. 

We will consider a function in the cluster algebra for $\Conf_4(\A)$ constructed in \cite{Le2}. Let $1 \leq a, b, c, d < n$ be four integers satisfying $a + b > n$ and $a+b+c+d=2n$. Then there an invariants inside $$[V_{\omega_a} \otimes V_{\omega_b}  \otimes V_{\omega_c} \otimes V_{\omega_d}]^{SL_n}$$ given by the web in below.

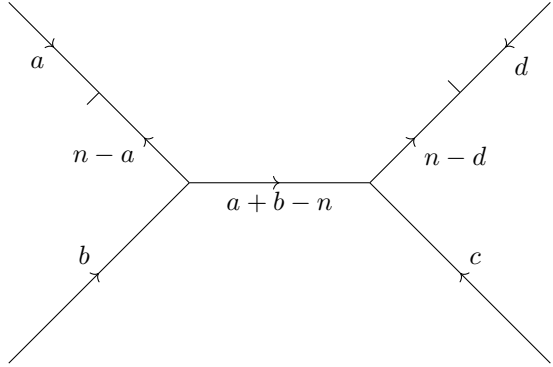
\vspace{5mm}

\begin{center}
\begin{tikzpicture}[scale=.8]
\begin{scope}[decoration={
    markings,
    mark=at position 0.5 with {\arrow{>}}},
    xshift=-8cm, yshift=-2cm
    ] 
\draw [postaction={decorate}] (-3,3) -- (-1.5,1.5) node [midway,below left] {$a$}; 
\draw [postaction={decorate}] (0,0) -- (-1.5,1.5) node [midway,below left] {$n-a$};
\draw [postaction={decorate}] (6,3) -- (4.5,1.5) node [midway,below right] {$d$};
\draw [postaction={decorate}] (3,0) -- (4.5,1.5) node [midway,below right] {$n-d$};
\draw [postaction={decorate}] (-3,-3) -- (0,0) node [midway, above left] {$b$};
\draw [postaction={decorate}] (6,-3) -- (3,0) node [midway, above right] {$c$};
\draw [postaction={decorate}] (0,0) -- (3,0) node [midway, below] {$a+b-n$};

\draw (4.3,1.7) -- (4.5,1.5) ;
\draw (-1.7,1.3) -- (-1.5,1.5) ;
\end{scope}

\end{tikzpicture}
\end{center}

Here is a more concrete description of the function. Given four flags 
$$t_1, \dots, t_n;$$
$$u_1, \dots, u_n;$$
$$v_1, \dots, v_n;$$
$$w_1, \dots, w_n;$$
first consider the forms
$$T_a := t_1 \wedge \cdots \wedge t_a,$$
$$U_b := u_1 \wedge \cdots \wedge u_b,$$
$$V_c := v_1 \wedge \cdots \wedge v_c,$$ 
$$W_d := w_1 \wedge \cdots \wedge w_d.$$
There is a natural map 
$$\phi_{a+b-n, n-a}: \bigwedge\nolimits^{b} V \rightarrow \bigwedge\nolimits^{a+b-n} V \otimes \bigwedge\nolimits^{n-a} V.$$
There are also natural maps 
$$W_d \wedge - \wedge V_c :  \bigwedge\nolimits^{a+c-n} V \rightarrow \bigwedge\nolimits^{n} V \simeq F$$ and 
$$T_a \wedge - : \bigwedge\nolimits^{n-a} V \rightarrow \bigwedge\nolimits^{n} V \simeq F.$$
Applying these maps to the first and second factors of $\phi_{a+c-n,n-a}(U_b)$, respectively, and then multiplying, we get get the value of our function. This is a function on $\Conf_4 \A_{SL_{n}}$. Let us call this function $F$.

\begin{conj} $F^t(x_1,x_2,x_3,x_4)$ is given by the minimum value of
$$d_a(p,x_1) + d_b(p,x_2) + d_{a+b-n}(q,p) + d_c(q,x_3) + d_d(q,x_4)$$
over $p$ and $q$ in the affine building for $PGL_n$.
\end{conj}

In other words, the function $F^t$ is given by the minimal weighted distance over a graph embedded in the building. The leaves of the graph are prescribed to land on the points $x_1, x_2, x_3, x_4$, while the weights are determined by the web calculating the function $F$.


\begin{thebibliography}{99}

\bibitem[FG1]{FG1} V.V. Fock, A.B. Goncharov. Moduli spaces of local systems and higher Teichmuller theory. Publ. Math. IHES, n. 103 (2006) 1-212.

\bibitem[GS]{GS} A.B. Goncharov, L. Shen. Geometry of canonical bases and mirror symmetry. Inventiones Math. 202, 487–633 (2015).

\bibitem[GS2]{GS2} A.B. Goncharov, L. Shen. Donaldson-Thomas trasnsformations of moduli spaces of G-local systems. arXiv:1602.06479

\bibitem[GHKK]{GHKK} M. Gross, P. Hacking, S. Keel, and M. Kontsevich. Canonical bases for cluster algebras. arXiv:1411.1394

\bibitem[K]{K} J. Kamnitzer. Hives and the fibres of the convolution morphism, Selecta Math. N.S. 13 no. 3 (2007), 483-496.

\bibitem[Le]{Le} I. Le. Higher Laminations and Affine Buildings. Geometry \& Topology 20, 1673–1735 (2016).

\bibitem[Le2]{Le2} I. Le. Cluster Algebras on Higher Teichmuller Spaces for Classical Groups. arXiv:1603.03523

\bibitem[LO]{LO} I. Le, E. O'Dorney. Geometry for Positive Configurations in Affine Buildings. arXiv:1511.00165

\bibitem[M]{M} A.G. Moshonkin. Concerning Hall's Theorem, from Mathematics in St. Petersburg, eds. A. A. Bolibruch, A.S. Merkur'ev, N. Yu. Netsvetaev. Amererican Mathematical Society Translations, Series 2, Volume 174, 1996.

\bibitem[R]{R} R. Rado. A theorem on independence relations, Quarterly J. Math. Oxford (2) 13 (1942), 83–89.

\end{thebibliography}
\end{document}